\pdfoutput=1
\documentclass{article}
\usepackage[utf8]{inputenc}
\usepackage{amsmath}
\usepackage{amssymb}
\usepackage{amsthm}
\usepackage{mathtools}
\usepackage{hyperref}
\usepackage{comment}
\usepackage[titletoc,title]{appendix}
\hypersetup{colorlinks=true}
\usepackage[totalwidth=480pt, totalheight=680pt]{geometry}

\title{Fully Faithful Functors and Dimension}
\author{Noah Olander}
\date{}

\newtheorem{thm}{Theorem}

\newtheorem{prop}{Proposition}
\newtheorem{lemma}{Lemma}

\theoremstyle{definition}
\newtheorem{definition}{Definition}

\newtheorem*{remark}{Remark}

\begin{document}

\maketitle 

\begin{abstract}
    We define the countable Rouquier dimension of a triangulated category and use this notion together with Theorem 2 of \cite{olander2021rouquier} to prove that if there is a fully faithful embedding $D^b_{coh}(X) \subset D^b_{coh}(Y)$ with $X, Y$ smooth proper varieties, then $\mathrm{dim}(X) \leq \mathrm{dim}(Y)$. 
\end{abstract}

We will show how Theorem \ref{thm-main} follows from \cite[Theorem 2]{olander2021rouquier}. Theorem \ref{thm-main} was expected to be true by Conjecture 10 of \cite{Orl09}, but to the best of the author's knowledge, it was unknown until now. The author is very grateful to Dmitrii Pirozhkov for suggesting this proof almost immediately after the paper \cite{olander2021rouquier} was posted. The author also thanks Dmitri Orlov for a helpful email. The key is the following definition due to Pirozhkov:

\begin{definition} Let $\mathcal{T}$ be a triangulated category. The \emph{countable Rouquier dimension of} $\mathcal{T}$, denoted $\mathrm{CRdim}(\mathcal{T})$, is the smallest $n$ such that there exists a countable set $\{E_i\}_{i \in I}$ of objects of $\mathcal{T}$ such that $\mathcal{T} = \langle \{E_i\}_{i \in I} \rangle _{n+1}$.
\end{definition}

We allow the countable Rouquier dimension to be infinity. We have an easy lemma:

\begin{lemma}
\label{lemma-admissible}
Let $\mathcal{A} \subset \mathcal{T}$ be an admissible subcategory in a triangulated category. Then $\mathrm{CRdim}(\mathcal{A}) \leq \mathrm{CRdim}(\mathcal{T})$.
\end{lemma}

\begin{proof}
Let $R$ be the right adjoint of the inclusion. If $\mathcal{T} = \langle \{E_i \}_{i \in I} \rangle _{n+1}$ then since $R$ is essentially surjective, $\mathcal{A} = \langle \{R(E_i)\}_{i \in I} \rangle _{n+1}$.
\end{proof}

The remark following the proof of \cite{olander2021rouquier} directly implies:

\begin{thm}
\label{thm-paper}
Let $X$ be a Noetherian regular scheme with affine diagonal. Then $\mathrm{CRdim}(D^b_{coh}(X)) \leq \mathrm{dim}(X)$. 
\end{thm}


In fact Theorem \ref{thm-paper} is true without the assumption on the diagonal of $X$ but we don't need it here. The reverse inequality is not true in general. For example, if $X$ is a variety over a countable field $k$, then $D^b_{coh}(X)$ has only countably many objects up to isomorphism, hence in fact $\mathrm{CRdim}(D^b_{coh}(X)) = 0$. However for varieties over $\mathbf{C}$ we do get the reverse inequality:

\begin{prop}
\label{prop-complete}
Let $k$ be an uncountable field. Let $X$ be a reduced scheme of finite type over $k$. Then $\mathrm{CRdim}(D^b_{coh}(X)) \geq \mathrm{dim}(X)$. 
\end{prop}

\begin{proof}
Compare to the proof of \cite[Proposition 7.17]{rouquier_2008}. Let $n = \mathrm{CRdim}(D^b_{coh}(X))$ and let $\{E_i\}_{i \in I}$ be a countable family of objects such that $D^b_{coh}(X) = \langle \{E_i\}_{i \in I} \rangle_{n+1}$. Consider the set of closed points $x \in X$ such that for every $i \in I$, the cohomology modules of $(E_i)_x$ are free $\mathcal{O}_{X,x}$-modules. Since a variety over $k$ is not a countable union of proper closed subsets (\cite[Exercise 2.5.10]{Liu}), the set contains a closed point $x$ such that $\mathrm{dim}(\mathcal{O}_{X,x}) = \mathrm{dim}(X)$. We have $(E_i)_x \in \langle \mathcal{O}_{X,x} \rangle_1$ for each $i$ since a complex with projective cohomology modules is decomposable, hence
$$
\kappa (x) \in \langle \{(E_i)_x\}_{i \in I} \rangle _{n+1} \subset \langle \mathcal{O}_{X,x} \rangle _{n+1},
$$
hence $n \geq \mathrm{dim}(X)$ by \cite[Proposition 7.14]{rouquier_2008}.
\end{proof}

\begin{remark}
Since the countable Rouquier dimension only gives the expected answer for varieties over a sufficiently large field, Orlov suggests an alternative notion of countable Rouquier dimension which makes sense for a $k$-linear dg-category $\mathcal{A}$ with $k$ a field: Take the smallest $n$ such that there exists a countable family $\{E_i\}_i$ of objects of $\mathcal{A}$ such that $\mathcal{A}_K = \langle \{(E_i)_K\}_i \rangle _{n+1}$ for every field extension $K/k$.
\end{remark}

Finally, we prove our main result:

\begin{thm}
\label{thm-main}
Let $k$ be a field. Let $X, Y$ be smooth proper varieties over $k$. Assume there exists a fully faithful, exact, $k$-linear functor $F: D^b_{coh}(X) \to D^b_{coh}(Y)$. Then $\mathrm{dim}(X) \leq \mathrm{dim}(Y)$.
\end{thm}

\begin{proof}
Choose any uncountable extension field $K/k$. By \cite[Theorem 1]{olander2020orlovs}, $F$ is the Fourier--Mukai transform with respect to a kernel $E \in D^b_{coh}(X \times _k Y)$. Then $E_K$ gives rise to a functor $F_K : D^b_{coh}(X_K) \to D^b_{coh}(Y_K)$ which remains fully faithful: The fact that $F$ is fully faithful may be rephrased as $R \circ F \cong \mathrm{id}$ where $R$ is the right adjoint of $F$. By the calculus of kernels, this may be rephrased as the existence of an isomorphism
\begin{equation}
\label{equation-isomorphism}
\mathbf{R}pr_{13 *} (\mathbf{L}pr_{12}^*(E) \otimes _{\mathcal{O}_{X \times Y \times X}} ^{\mathbf{L}} \mathbf{L} pr_{23}^* (E')) \cong \mathcal{O}_{\Delta}
\end{equation}
where $E' = \mathbf{R} \mathcal{H}om _{\mathcal{O}_{X \times Y}} (E, \mathbf{L}pr_1 ^* (\omega _X))[\mathrm{dim}(X)]$, viewed as an object of $D^b_{coh}(Y\times _k X)$, is the kernel of $R$, and (\ref{equation-isomorphism}) remains valid upon base change to $K$.

Now $F_K$ is the inclusion of an admissible subcategory since $X,Y$ are smooth and proper (\cite[\href{https://stacks.math.columbia.edu/tag/0FYN}{Tag 0FYN}]{stacks-project}). We have $\mathrm{CRdim}(D^b_{coh}(X_K)) = \mathrm{dim}(X_K) = \mathrm{dim}(X)$ by Theorem \ref{thm-paper} and Proposition \ref{prop-complete} and similarly for $Y$. Thus by Lemma \ref{lemma-admissible}, we have
$$
\mathrm{dim}(X) = \mathrm{CRdim}(D^b_{coh}(X_K)) \leq \mathrm{CRdim}(D^b_{coh}(Y_K)) = \mathrm{dim}(Y),
$$
as needed.
\end{proof}

\bibliographystyle{alpha}
\bibliography{references}

\textsc{Columbia University Department of Mathematics, 2990 Broadway, New York, NY 10027}

\href{mailto:nolander@math.columbia.edu}{nolander@math.columbia.edu}

\end{document}